\documentclass[12pt,reqno]{article}
\usepackage{amsmath,amssymb,amsbsy,amsfonts,amsthm,latexsym}
\usepackage{hyperref}   
\usepackage[margin=3cm]{geometry}         
\begin{document}
\newtheorem{theorem}{Theorem}
\newtheorem{lemma}[theorem]{Lemma}
\def\\{\cr}
\newcommand{\seqnum}[1]{\href{https://oeis.org/#1}{\rm \underline{#1}}}

\begin{center}
\vskip 1cm{\LARGE\bf 
Binary Recurrences for which Powers of Two are Discriminating Moduli
}
\vskip 1cm
\end{center}

\begin{minipage}[t]{0.5\textwidth}
\begin{center}
Adriaan A. de Clercq\\
Department of Mathematics and Applied Mathematics, University of Pretoria\\
Private Bag X20, Hatfield 0028\\ 
South Africa\\
\href{mailto:adriaandeclercq1998@gmail.com}{\tt
adriaandeclercq1998@gmail.com}
\\

\ \\
Florian Luca\\
School of Maths\\ Wits University\\1 Jan Smuts Avenue,
Braamfontein 2000,
Johannesburg\\
South Africa\\
Centro de Ciencias Matemat\'aticas, UNAM, Morelia\\ Mexico\\
\href{mailto:florian.luca@wits.ac.za}{\tt florian.luca@wits.ac.za}
\\
\ \\
Lilit Martirosyan\\
Department of Mathematics and Statistics\\
University of North Carolina, Wilmington\\
601 South College Road\\ Wilmington, NC 28403-5970\\ USA\\
\href{mailto:martirosyanl@uncw.edu}{\tt martirosyanl@uncw.edu}\\
\end{center}
\end{minipage}
\begin{minipage}[t]{0.5\textwidth}
\begin{center}
Maria Matthis\\
Department of Mathematics\\
Katharineum zu L\"ubeck\\
K\"onigsstra{\ss}e 27-31, 23552 L\"ubeck\\
Germany\\
\href{mailto:matthima@katharineum.de}{\tt matthima@katharineum.de}
\\
\ \\
Pieter Moree\\
Max-Planck-Institut f\"ur Mathematik\\
Vivatsgasse 7, D-53111 Bonn\\
Germany\\ 
\href{mailto:moree@mpim-bonn.mpg.de}{\tt moree@mpim-bonn.mpg.de} \\
\ \\
Max A. Stoumen\\
Department of Mathematics and Statistics \\
University of North Carolina, Wilmington\\
601 South College Road\\ Wilmington, NC 28403-5970\\
USA\\
\href{mailto:mas5084@uncw.edu}{\tt mas5084@uncw.edu} \\
\ \\
Melvin Wei\ss\\
Department of Mathematics \\
Universit\"at Bonn\\
Endenicher Allee 60, 53115 Bonn \\
Germany\\
\href{mailto:melvin@weissprivat.de}{\tt melvin@weissprivat.de} \\

\end{center}
\end{minipage}
\begin{abstract} Given a sequence 
${\bf w}=(w_n)_{n\geq 0}$ of distinct positive integers $w_0 , w_1, w_2, \ldots$ and any positive integer $n$, we define the discriminator function $\mathcal{D}_{\bf w}(n)$ to be the smallest positive integer $m$ such 
that $w_0,\ldots, w_{n-1}$ are pairwise incongruent modulo $m$. In this paper,  we classify all binary recurrent sequences $\bf w$ consisting of different integer terms such that $\mathcal{D}_{\bf w}(2^e)=2^e$  for every $e\geq 1$. For all
of these sequences 
it is expected that one can actually give a fairly simple description
of
$\mathcal{D}_{\bf w}(n)$ for every $n\ge 1$. For one infinite 
family of such sequences this
has been done by Faye, Luca, and Moree, and for another by Ciolan and Moree.
\end{abstract}

\section{Introduction}

The {\it discriminator sequence} of a sequence ${\bf w}=(w_n)_{n\ge 0}$
of distinct integers is the sequence 
$({\mathcal D}_{\bf w}(n))_{n\ge 0}$
given by
$$
{\mathcal D}_{\bf w}(n)=\min\{m\ge 1: w_0,\ldots,w_{n-1}~{\text{\rm are~pairwise~distinct~modulo}}~m\}.
$$
In other words,
${\mathcal D}_{\bf w}(n)$
is the smallest
integer $m$ that allows one to discriminate (tell
apart) the integers $w_0,\ldots, w_{n-1}$  on reducing them modulo $m$.
If not all integers are distinct, but say $w_0,\ldots,w_k,$ then we
can define ${\mathcal D}_{\bf w}(j)$ for $j=1,\ldots,k+1$.
Obviously ${\mathcal D}_{\bf w}(n)$ is non-decreasing as a function of $n$.
Note that since $w_0,\ldots,w_{n-1}$ are in $n$ distinct residue classes modulo ${\mathcal D}_{\bf w}(n)$, we must have ${\mathcal D}_{\bf w}(n)\ge n$.
On the other hand clearly
$${\mathcal D}_{\bf w}(n)\le \max\{w_0,\ldots,w_{n-1}\}
-\min\{w_0,\ldots,w_{n-1}\}+1.$$
The main problem is to give an easy description or 
characterization of ${\mathcal D}_{\bf w}(n)$ (in many cases such a characterization does not seem to exist). 

If $w_j$ is a polynomial in
$j,$ the behavior of the discriminator is fairly well
understood. See Moree \cite{PM},
Zieve \cite{Zieve}, and the references 
therein.  

An intensively studied class of
sequences is that of binary 
recurrent sequences, cf.\ the book by 
Everest et al.\ \cite{E}. For a generic
binary recurrent sequence there is 
currently no meaningful 
characterization of 
its discriminator. An example
is provided by the discriminator for the Fibonacci sequence (see Table 1).
However, if we have 
\begin{equation}
\label{condition}
\mathcal{D}_{\bf w}(2^e)=2^e\quad {\text{\rm for every}}\, \, e\geq 1, 
\end{equation}
the discriminator behavior tends to
be much simpler. It is easy to see that then $\mathcal{D}_w(n)<2n$.  
This allows one to exclude many potential discriminator values. Indeed, in
general discriminator characterizations  for 
a fixed $n$ proceed by excluding  all
integers different from $\mathcal{D}_{\bf w}(n)$ as values.
If \eqref{condition}  holds, then typically many powers of two
occur as values (cf.\ Table 2). 
All known binary recurrent discriminators 
satisfy \eqref{condition} and fall into two families
described below.
Thus, it is natural to ask for a classification
of all binary recurrent sequences $(w_n)_{n\geq 0}$ such that \eqref{condition} is satisfied. Note that for any such sequence the terms $w_n$ must be
distinct.

Our main result completely answers this question.
\begin{theorem}
\label{main}
For integers $w_0,w_1,p$ and $q$, let $(w_n)_{n\ge 0}$ be 
the sequence defined by
\begin{equation}
\label{eq1}
w_{n+2}=pw_{n+1}+qw_n\quad {\text{\rm for~all}}\, \, n\ge 0.
\end{equation}

\noindent If $w_0+w_1$ is even and $k\ge 1$, then 
$\#\{w_n\,({\rm mod~}2^k): 0\le n\le 2^k-1\}<2^k$.

\noindent If $(p\bmod 4,q\bmod 4)=(2,3)$ and $w_0+w_1$ is odd, then ${\mathcal D}_{\bf w}(2^k)=2^k$ for every 
$k\ge 1$. 

\noindent If $(p\bmod 4,q\bmod 4)\neq (2,3)$ and $k\ge 3$, then $\#\{w_n\,({\rm mod~}2^k): 0\le n\le 2^k-1\}<2^k$.
\end{theorem}
Representing the residue classes modulo 
$m$ by $\overline{a}$, with 
$0\le a\le m-1$, we can reformulate property
\eqref{condition} as saying that the
map from $\mathbb Z/m\mathbb Z$ to $\mathbb Z/m\mathbb Z$ given by $\overline{a}\mapsto \overline{u_a}$ is a \emph{permutation} for every $m$ that is a power of two.\hfil\break
\indent We next describe the binary recurrent sequences for which the discriminator has been characterized. They fall into two
families. Theorem \ref{main} shows at
a glance that for all of them \eqref{condition} is satisfied.\hfil\break
\par Family 1. In Faye et al.\ \cite{FLM}, and its continuation 
by Ciolan et al.\ \cite{CFLM}, the discriminator
${\mathcal D}_{{\bf U}(k)}(n)$ is studied, where the \emph{Shallit sequence} ${\bf U}(k)$ is given by ${\bf U}(k)=(U_n(k))_{n\ge 0}$ with $U_0(k)=0,~U_1(k)=1$ and 
$$
U_{n+2}(k)=(4k+2)U_{n+1}(k)-U_n(k)
$$
for all $n\ge 0$. By Theorem \ref{main}, we have
${\mathcal D}_{{\bf U}(k)}(2^e)=2^e$ for every $e\ge 1$.\hfil\break
\par Family 2.
Let $q\ge 5$ be a prime and put
$q^*=(-1)^{(q-1)/2}\cdot q$. 
The
sequence $u_q(1),u_q(2),\ldots,$ with 
$$u_q(j)=\frac{3^j-q^*(-1)^j}{4},$$
we call the \emph{Browkin-S{\u a}l{\u a}jan sequence} for $q$.
The sequence $u_q$ satisfies the
recursive relation $u_q(j)=2 u_q(j-1) + 3 u_q(j-2)$ for $j \geq 3,$ with initial 
values 
$$ 
u_q(1)=(3+q^*)/4 \quad {\text{\rm and}}\quad  u_q(2)=(9-q^*)/4.
$$ 
We denote its discriminator by ${\mathcal D}_q$.
In the context of the
discriminator, the sequence $u_5$ ($2, 1, 8, 19, 62, 181, 548, 1639, 4922,\ldots$) was first considered 
by Sabin S{\u a}l{\u a}jan during an internship carried out in 2012 under the guidance of Moree (for
this reason we call it the \emph{S{\u a}l{\u a}jan sequence}).
Disregarding signs this is sequence \seqnum{A084222}.
Moree and Zumalac\'arregui \cite{PA} 
determined ${\mathcal D}_5(n)$ (cf.\ Table 2).
\begin{theorem}
\label{mainAna}
Let $n\ge 1$ be an arbitrary integer. Let $e$ be the smallest integer such that
$2^e\ge n$ and $f$ be the smallest integer such that $5^f\ge 5n/4$.
Then ${\mathcal D}_5(n)=\min\{2^e,5^f\}$.
\end{theorem}
More recently Ciolan and Moree \cite{CM} completely 
characterized ${\mathcal D}_q$
for arbitrary primes $q>5$. Noting that $u_q(1)+u_q(2)=3$,
one sees that Theorem \ref{main} applies and hence
${\mathcal D}_q(2^e)=2^e$ for every $e\ge 1$.

In order to prove Theorem \ref{main}, we will deal
with the special case where ${\bf w}$ is a Lucas sequence first in 
Section \ref{secLucas}.
In the general case, we express ${\bf w}$ as
a linear combination of a Lucas and a shifted Lucas sequence (Section 
\ref{sec:general}). Our arguments require some consideration of the
two divisibility of binomial coefficients (Section \ref{secbinomial}).

Beyond the
polynomial and the recurrence
sequence case there is very little
known. Haque and Shallit 
\cite{HS} considered the
discriminator for $k$-regular
sequences. For these also 
property \eqref{condition} is satisfied. Sun \cite{Sun} made
some conjectures regarding the
discriminator for various
sequences.
\section{Preliminaries}
\label{secbinomial} 
We recall a celebrated result of Kummer, cf.\  Ribenboim \cite[pp. 30--33]{Ribenboim}.
\begin{theorem}[Kummer, 1852]
\label{Kummer}
Let $p$ be a prime number. The exponent of $p$ in ${n \choose m}$ is the number of base $p$ carries when summing $m$ with $n-m$ in base $p$. 
\end{theorem}
Here and in what follows we write $\nu_2(a)$ for the exponent of $2$ in the factorization of 
the integer $a$. 
\begin{lemma}
\label{binomial}
We have 
$$
\nu_2\left({{\ell}\choose{k}} 2^{3k}\right)>\nu_2(2\ell)
$$
for all $k\ge 1$. Further, 
$$
\nu_2\left({{2^k}\choose{\ell}} 2^{\ell}\right)\ge k+3
$$
for $\ell=3$ and $\ell\ge 5$. 
\end{lemma} 

\begin{proof}
We use Theorem \ref{Kummer} with $p=2$.
For the first inequality, we note that it is clear for $k=1$, so we may assume that $k\ge 2$. Write $\ell=2^{\ell_0}\ell_1$ with integers $\ell_0\ge 0$ and $\ell_1$ odd. The inequality is clear for 
$\ell_0\le 1$. It is also clear if $k>(\ell_0+1)/3$. So, we may assume that $k\le (\ell_0+1)/3$. Write $k=2^{k_0}k_1$, where $k_0\ge 0$ and $k_1$ is odd. Then $k_0<k\le (\ell_0+1)/3<\ell_0$.
It follows that by summing up $k$ with $\ell-k$, we have at least $\ell_0-k_0$ carries in base $2$. Thus, 
$$
\nu_2\left({{\ell}\choose{k}} 2^{3k}\right)\ge (\ell_0-k_0)+3k>\ell_0+2k\ge \ell_0+2,
$$
which is what we wanted to prove. 

We will now prove the second inequality. Assume first that $\ell \in [3,2^k-1]$. Then the number of carries from summing up $\ell$ with $2^k-\ell$ is, by the previous argument, $k-\ell_0$, where again $\ell=2^{\ell_0}\ell_1$ 
with $\ell_1$ odd. Hence,
$$
\nu_2\left({{2^k}\choose{\ell}} 2^{\ell}\right)=k-\ell_0+\ell.
$$
This is at least $k+3$ if $\ell\ge 3$ is odd (since then $\ell_0=0$). It is also at least $k-\ell_0+2^{\ell_0}>k+3$ if $\ell_0\ge 3$. If $\ell_0=1$, then $\ell>4$ so $k-\ell_0+\ell\ge k+3$. Finally, if $\ell_0=2$, then since $\ell\ne 4$, we have
$\ell\ge 8$ (since $4\mid \ell$), so the above expression is at least $k-2+8>k+3$. This was for $\ell<2^k$. Finally, when $\ell=2^k$, we have 
$$
\nu_2\left({{2^k}\choose{\ell}} 2^{\ell}\right)=2^k> k+3
$$
because $k\ge 3$. 
\end{proof}
\section{The Lucas sequence}
\label{secLucas}
A basic role in the theory of binary recurrent sequences is played by Lucas sequences.
\begin{theorem}
\label{thmLucas}
Let $(u_n)_{n\ge 0}$ be a Lucas sequence with $u_0=0,~u_1=1$ and 
$$u_{n+2}=pu_{n+1}+qu_n,\quad {\text{\rm for~all}}\, \, n\ge 0.$$ Then ${\mathcal D}_{\bf u}(2^k)=2^k$ for all $k\ge 1$ if and only if 
if $(p\bmod{4},q\bmod{4})=(2,3)$.
\end{theorem}

\begin{proof}
We look at $\{u_0,u_1,u_2,u_3\}=\{0,1,p,p^2+q\}$. Since these are all the residues modulo $4$, it follows that either $(p\bmod 4, q\bmod 4)=(2,3)$ or $(p\bmod 4, q\bmod 4)=(3,1)$. The second possibility entails 
$(p\bmod 8, q\bmod 8)\in \{(3,1),~(7,1),~(3,5),~(7,5)\}$ and one checks computationally that none  of these $4$ possibilities gives that $\{u_k\,({\rm mod~}8): 0\le k\le 7\}$ covers all residue classes modulo $8$. 
Thus, we must have $(p\bmod 4, q\bmod 4)=(2,3)$. 

We consider the quadratic polynomial $x^2-px-q$ 
having discriminant $\Delta=p^2+4q$. The 
equation $x^2-px-q=0$ is the characteristic equation
for the Lucas sequence. We consider 
the cases $\Delta=0$ and $\Delta\ne 0$ separately.

The degenerate case ($\Delta=0$). 
In this case $u_n=np_0^{n-1}$ with 
$p_0=p/2$. We have 
$\{u_0,u_1,u_2,u_3\}=\{0,1,2p_0,3p_0^2\}$ and 
since $p_0$ is odd,
these are distinct modulo $4$. We claim that 
$\nu_2(u_m-u_n)=\nu_2(m-n)$ 
for $m>n$.
Notice that this claim implies \eqref{condition}.

We have $u_m-u_n\equiv m-n\,({\rm mod~}2)$.
So $\nu_2(u_m-u_n)=0$ if and only if $\nu_2(m-n)=0$.
 Next assume that $m\equiv n\,({\rm mod~}2)$. 
Write $m=n+2\ell$. Then 
\begin{equation}
\label{difference}
u_m-u_n=(n+2\ell)p_0^{n+2\ell-1}-np_0^{n-1}=(n+2\ell)p_0^{n-1}(p_0^{2\ell}-1)+2\ell p_0^{n-1}.
\end{equation}
We can write 
$p_0^2=1+8p_1$ with $p_1$ an integer.
Thus,
$$
p_0^{2\ell}=(1+8p_1)^{\ell}=1+8\ell p_1+{{\ell}\choose{2}} (8p_1)^2+\cdots.
$$
This in combination with \eqref{difference} leads to 
$$
u_m-u_n = p_0^{n-1}\left(2\ell+\sum_{k\ge 1} (n+2\ell) {{\ell}\choose{k}} (8p_1)^k\right).
$$
Since by Lemma \ref{binomial}  for every $k\ge 1$ we have 
$$
\nu_2\left({{\ell}\choose {k}} (8p_1)^k\right)> \nu_2(2\ell),
$$
we conclude that
$$
\nu_2(u_m-u_n)=\nu_2(2\ell)=\nu_2(m-n),
$$
thus establishing the claim. 

The non-degenerate case ($\Delta\ne 0$). Since $p=2p_0$ and $q\bmod 4=3$, it follows that $\Delta=4(p_0^2+q)=16\Delta_0$, where $\Delta_0$ is an integer.  Let
$$
\alpha=p_0+2{\sqrt{\Delta_0}}\qquad {\text{\rm and}}\qquad  \beta=p_0-2{\sqrt{\Delta_0}}
$$
be the roots of $x^2-px-q=0$. 
The Binet formula for $u_n$ is 
\begin{equation}
\label{eq11}
u_n=\frac{\alpha^n-\beta^n}{\alpha-\beta}.
\end{equation}
While not necessary for this proof, we take a detour and prove a property concerning the index of appearance of powers of $2$. In the course of proving it, we will show that $2\| v_n$, with 
$(v_n)_{n\ge 0}$ the companion sequence of our Lucas sequence. This fact we actually do need in
our proof.

For a positive integer $m$ let $z(m)$ be the order of appearance of $m$ in the sequence $(u_n)_{n\ge 0}$. It is the minimal positive integer $k$ such that 
$m\mid u_k$. It is known 
(see Bilu et al.\ \cite{Bilu}) that this exists for all $m$ which are coprime to $q$. Further, $m\mid u_n$ if and only if $z(m)\mid n$. For us, $z(2)=2$ since $p\equiv 2\,({\rm mod~}4)$ and $z(4)=4$. 
It follows easily by induction on $k$ that 
$$
z(2^k)=2^k.
$$
One way to see this is to introduce the companion sequence $(v_n)_{n\ge 0}$ given by $v_0=2,~v_1=p$ and $v_{n+2}=pv_{n+1}+qv_n$ for all $n\ge 0$. By induction, we get that $2\| v_n$ 
for all $n\ge 0$. The Binet formula for 
$v_n$ is 
\begin{equation}
\label{eq12}
v_n=\alpha^n+\beta^n\quad {\text{\rm for~all~}}n\ge 0.
\end{equation}
We have $u_{2n}=u_nv_n$ by the Binet formulas \eqref{eq11} and \eqref{eq12}. We are now ready to show that $z(2^k)=2^k$. Assume that $k\ge 3$ and that $2^k\mid u_n$. This implies that $n=2^{\ell}n_1$ 
for some integers $\ell$ and $n$, 
with $\ell\ge 2$ and $n_1$ odd. Now we use repeatedly the formula 
$u_{2m}=u_mv_m$ for $m=n/2,~n/4,\ldots,$ resulting in
$$
u_n=u_{2^\ell n_1}=v_{2^{\ell-1}n_1}v_{2^{\ell-2} n_1}\cdots v_{n_1} u_{n_1}.
$$
Since $v_{2^i n_1}\equiv 2\,({\rm mod~}4)$ for $i=0,1,\ldots,\ell-1$ 
and $u_{n_1}$ is odd, we infer that 
$$
\nu_2(u_{2^{\ell} n_1})=\ell.
$$
It follows that $k\ge \ell$. In particular, $z(2^k)=2^k$. 

Next we show that 
\begin{equation}
\label{twotopowerkshift}
u_{n+2^k}\equiv u_n+2^k\,({\rm mod~}{2^{k+1}})
\end{equation}
for all $k\ge 1$. One checks it easily by hand for $k=1$ and $n=0,1$ as well as for $k=2$ and $n=0,1,2,3$. Assume next $k\ge 3$. In what follows, for three algebraic integers $a,b,c$, we write $a\equiv b\,({\rm mod~}c)$ if 
$(a-b)/c$ is an algebraic integer. 
We have 
\begin{align*}
\alpha^{2^k}  = (p_0+2{\sqrt{\Delta_0}})^{2^k}=& p_0^{2^k}+2^k p_0^{2^k-1}(2{\sqrt{\Delta_0}})+{{2^k}\choose {2}} p_0^{2^k-2} (2{\sqrt{\Delta_0}})^2\\
 + & {{2^k}\choose {4}} p_0^{2^k-4} (2{\sqrt{\Delta_0}})^4+\sum_{\substack{\ell\ge 3\\ \ell\ne 4}} {{2^k}\choose {\ell}} p_0^{2^k-\ell} (2{\sqrt{\Delta_0}})^{\ell}.
\end{align*}
Then, by Lemma \ref{binomial},
$$
\alpha^{2^k}\equiv p_0^{2^k}+2^k p_0^{2^k-1}(2{\sqrt{\Delta_0}})+{{2^k}\choose {2}} p_0^{2^k-2} (2{\sqrt{\Delta_0}})^2+{{2^k}\choose{4}} p_0^{2^k-1} (2{\sqrt{\Delta_0}})^4 \,({\rm mod~}{2^{k+3}{\sqrt{\Delta_0}}}).
$$
Changing $\alpha$ to $\beta$, the same calculation yields
$$
\beta^{2^k}\equiv p_0^{2^k}-2^k p_0^{2^k-1}(2{\sqrt{\Delta_0}})+{{2^k}\choose{2}} p_0^{2^k-2} (2{\sqrt{\Delta_0}})^2 +{{2^k}\choose{4}} p_0^{2^k-4} (2{\sqrt{\Delta_0}})^4\,({\rm mod~}{2^{k+3}{\sqrt{\Delta_0}}}).
$$
Thus,
\begin{align*}
\alpha^{n+2^k}-\beta^{n+2^k} & \equiv \alpha^n\left(p_0^{2^k}+2^k p_0^{2^k-1}(2{\sqrt{\Delta_0}})+{{2^k}\choose {2}} p_0^{2^k-2} (2{\sqrt{\Delta_0}})^2+{{2^k}\choose {4}} p_0^{2^k-4} (2{\sqrt{\Delta_0}})^4\right)\\
& -  \beta^n\left(p_0^{2^k}-2^k p_0^{2^k-1}(2{\sqrt{\Delta_0}})+{{2^k}\choose{2}} p_0^{2^k-2} (2{\sqrt{\Delta_0}})^2+{{2^k}\choose{4}} p_0^{2^k-4} (2{\sqrt{\Delta_0}})^4\right)\\
& \equiv  p_0^{2k}(\alpha^n-\beta^n)+2^k p_0^{2^k-1} (2{\sqrt{\Delta_0}}) (\alpha^n+\beta^n)\\
& + {{2^k}\choose{2}} p_0^{2^k-2} (2{\sqrt{\Delta_0}})^2 (\alpha^n-\beta^n) \\
& +  {{2^k}\choose{4}} p_0^{2^k-4} (2{\sqrt{\Delta_0}})^4(\alpha^n-\beta^n)\,({\rm mod~}{2^{k+3}{\sqrt{\Delta_0}}}).
\end{align*}
Dividing across by $\alpha-\beta$ (which is equal to 
$4{\sqrt{\Delta_0}}$), we obtain
\begin{equation}
\begin{aligned}
\label{eq101}
u_{n+2^k} &\equiv p_0^{2^k} u_n+2^k p_0^{2^k-1} (v_n/2)+{{2^k}\choose{2}} p_0^{2^k-2} (4\Delta_0) u_n\\
& + {{2^k}\choose{4}} p_0^{2^k-4} (16\Delta_0^2) u_n\,({\rm mod~}{2^{k+1}}).
\end{aligned}
\end{equation}
We have $p_0^{2^k}\equiv 1\,({\rm mod~}{2^{k+1}})$ 
and $v_n/2 \equiv 1\,({\rm mod~}2)$. Finally, 
$$
{{2^k}\choose{2}} p_0^{2^k-2} (4\Delta_0)=2^{k+1}(2^k-1)p_0^{2^k-2} \Delta_0\equiv 0\,({\rm mod~}{2^{k+1}}),
$$
and also
$$
{2^k\choose 4} p_0^{2^k-4} (16\Delta_0^2)=\frac{2^{k-2}(2^k-1)(2^{k-1}-1)(2^k-3)}{3} 2^4\Delta_0^2\equiv 0\,({\rm mod~}{2^{k+1}}).
$$
We thus get from \eqref{eq101} that \eqref{twotopowerkshift}
holds for all $k\ge 1$. This implies by induction on $k$ that 
${\mathcal D}_{\bf u}(2^k)=2^k$. 
\end{proof}

\section{The general case: the proof of Theorem \ref{main}}
\label{sec:general}

In the previous section we dealt with the Lucas sequence (Theorem \ref{thmLucas}). 
We will make crucial use of that result in order to deal with a more general 
recurrence $(w_n)_{n\ge 0}$ as in \eqref{eq1}.

\begin{proof}[Proof of Theorem \ref{main}]
If $\#\{w_n\,({\rm mod~}{2^k}): 0\le n\le 2^k-1\}=2^k$ for all $k$, it holds for $k=1$ in particular. Thus, $w_0,~w_1$ have different parities which is equivalent to $w_0+w_1$ being odd. This proves the first assertion. Conversely, write 
$$
w_n=a u_n+bu_{n+1}.
$$
Note that $a u_n+bu_{n+1}$ satisfies the same recurrence 
relation as $w_n$.
On setting $n=0$, respectively $n=1$, we find $b=w_0$ 
and $a=w_1-pw_0$. Thus, $a+b=(w_1+w_0)-pw_0$ is odd. 
By \eqref{twotopowerkshift}, we obtain
\begin{align*}
w_{n+2^k} & =  au_{n+2^k}+bu_{n+1+2^k}\equiv  a(u_n+2^k)+b(u_{n+1}+2^k)\\
& \equiv   (au_n+bu_{n+1})+(a+b)2^k \equiv  w_n+2^k
\,({\rm mod~}{2^{k+1}})
\end{align*}
for $k\ge 1$. This shows that ${\mathcal D}_{\bf w}(2^k)=2^k$ for every $k\ge 1$.

It remains to prove the final assertion. 
Note that it is enough to prove it for $k=3$.
This can be done by doing a computer calculation modulo $8$. We consider all integers $a,b,p,q$ with
$0\le a,b,p,q\le 7$ and compute $\#\{w_n\,({\rm mod~}8): 0\le n\le 7\}$. It turns out that if $(p\bmod 4,q\bmod 4)\neq (2,3)$, 
then this cardinality is $<8$.
\end{proof}

\section{Tables}
We tabulate the discriminator for a sequence that does not (Fibonacci
sequence) and
a sequence that does (S{\u a}l{\u a}jan sequence) satisfy the conditions of Theorem \ref{main}.
We give the prime factorization of the values. Note the big difference in behavior.

\bgroup
\def\arraystretch{1.05}
\begin{table}[h]
\begin{center}
\begin{tabular}{|c|c|c|c|c|c|}
\hline
$n$ & $D_F(n)$ & $n$ & $D_F(n)$ & $n$ & $D_F(n)$ \\
\hline\hline
$1$ & $1$ & $21 - 24$ & $59$ & $69 - 80$ & $431$ \\
\hline
$2$ & $2$ & $25 - 26$ & $79$ & $81 - 113$ & $3 \cdot 197$ \\
\hline
$3$ & $3$ & $27 - 32$ & $83$ & $114 - 115$ & $3 \cdot 283$ \\
\hline
$4$ & $5$ & $33 - 35$ & $2^{3} \cdot 3 \cdot 5$ & $116 - 152$ & $1039$ \\
\hline
$5$ & $2^{3}$ & $36 - 39$ & $157$ & $153 - 158$ & $5 \cdot 13 \cdot 17$ \\
\hline
$6$ & $3^{2}$ & $40 - 44$ & $173$ & $159 - 162$ & $1171$ \\
\hline
$7 - 8$ & $2 \cdot 7$ & $45 - 55$ & $193$ & $163 - 166$ & $1451$ \\
\hline
$9 - 11$ & $3 \cdot 5$ & $56 - 59$ & $311$ & $167 - 184$ & $3 \cdot 487$ \\
\hline
$12 - 16$ & $2 \cdot 3 \cdot 5$ & $60 - 64$ & $337$ & $185 - 208$ & $1609$ \\
\hline
$17 - 20$ & $5 \cdot 7$ & $65 - 68$ & $409$ & $209 - 281$ & $3 \cdot 761$ \\
\hline
\end{tabular}
\end{center}
\medskip
\caption{\seqnum{A270151}: Discriminator for the Fibonacci sequence
$1,2,3,5,8,13,\ldots$}
\end{table}
\egroup

\bgroup
\begin{table}[h]
\begin{center}
\begin{tabular}{|c|c|c|c|}
\hline
$n$ & $D_S(n)$ & $n$ & $D_S(n)$ \\
\hline\hline
$1$ & $1$ & $129-256$ & $2^8$\\
\hline
$2$ & $2$ & $257-512$ & $2^9$\\
\hline
$3-4$ & $2^2$ & $513-1024$ & $2^{10}$\\
\hline
$5-8$ & $2^3$ & $1025-2048$ & $2^{11}$\\
\hline
$9-16$ & $2^4$ & $2049-2500$ & $5^5$\\
\hline
$17-20$ & $5^2$ & $2501-4096$ & $2^{12}$\\
\hline
$21-32$ & $2^5$ & $4097-8192$ & $2^{13}$\\
\hline
$33-64$ & $2^6$ & $8193-12500$ & $5^6$\\
\hline
$65-100$ & $5^3$ & $12501-16384$ & $2^{14}$\\
\hline
$101-128$ & $2^7$ & $16385-32768$ & $2^{15}$\\
\hline
\end{tabular}
\end{center}
\caption{Discriminator for the S{\u a}l{\u a}jan sequence  $2, 1, 8, 19, 62, 181,\ldots$}
\end{table}
\egroup
Table 2 demonstrates Theorem \ref{mainAna}.

\vfil\eject
\section{Acknowledgments} The problem of characterizing the binary recurrences ${\bf w}$ such that
${\mathcal D}_{\bf w}(2^e)=2^e$ for every $e\ge 1$ was posed by Moree to many interns. Eventually it was 
solved, in essence, independently by de Clercq, Matthis, and Weiss 
(interns in 2019) and Stoumen (a student of Martirosyan). Their work was 
completed by Luca. Ciolan kindly commented on an earlier version.

The authors, except Stoumen who has not visited the Max Planck Institute for Mathematics (yet!), are grateful for the pleasant working conditions and inspiring research atmosphere.

{\it 2010 Mathematics Subject Classification}: Primary 11B50, Secondary 11A07, 11B39.

\noindent Keywords: congruence, discriminator, binary recurrent sequence, Lucas sequence, Binet formula, permutation.

\noindent (Concerned with sequences \seqnum{A084222} and \seqnum{A270151}.)
\end{document}